\documentclass{amsart}

\usepackage{amsfonts}
\usepackage{amssymb}
\usepackage{amsthm}
\usepackage{eucal}
\usepackage{rotating}

\theoremstyle{plain}
\newtheorem{theorem}{Theorem}[section]

\theoremstyle{definition}

\theoremstyle{remark}

\newcommand{\C}{\mathcal{C}}

\begin{document}
\title[Tutte-Coxeter graph]{A characterization of Tutte-Coxeter graph}

\author{A. Mohammadian}
\address{Department of Pure Mathematics, Ferdowsi University of Mashhad, Mashhad, Iran}
\email{abbasmohammadian1248@gmail.com}

\author{M. Farrokhi D. G.}
\address{Institute for Advanced Studies in Basic Sciences (IASBS), and the Center for Research in Basic Sciences and Contemporary Technologies, IASBS, P.O.Box 45195-1159, Zanjan 66731-45137, Iran}
\email{m.farrokhi.d.g@gmail.com, farrokhi@iasbs.ac.ir}

\date{}

\begin{abstract}
We give a natural generalization of the Tutte-Coxeter graph in a natural way and prove that the Tutte-Coxeter graph is the only vertex-transitive (edge-transitive) graph among all generalized Tutte-Coxeter graphs.
\end{abstract}
\maketitle
\section{introduction}
In 1969, a class of generalized Petersen graphs was defined by Watkins \cite{mew} as follows: for positive integers $n\geq 3$ and $1\leq k<n/2$, the generalized Petersen graph $P(n,k)$ is defined on a set of $2n$ vertices $u_0,u_1,\ldots,u_{n-1},v_0,v_1,\ldots,v_{n-1}$ whose edges are $\{u_i,u_{i+1}\}$, $\{u_i,v_i\}$ and $\{v_i,v_{i+k}\}$, where $i$ ranges over $\{1,2,\ldots,n\}$ and that all indices are taken modulo $n$. These graphs were studied earlier by Coxeter \cite{hsmc} and Robertson \cite{nr} in the special cases where $n$ and $k$ are relatively prime, and $k=2$, respectively.

The automorphism group of generalized Petersen graphs are determined completely in \cite{rf-jeg-mew}. It is shown that, with the exception of the dodecahedron $P(10,2)$, the graph $P(n,k)$ is vertex-transitive if and only if $k^2\equiv\pm1$ (mod $n$). Furthermore, $P(n,k)$ is a Cayley graph if and only if $k^2\equiv1$ (mod $n$), see \cite{ml,rn-ms}. Also, it is shown in \cite{rf-jeg-mew} that $P(n,k)$ is edge-transitive if and only if
\[(n,k)\in\{(4,1),(5,2),(8,3),(10,2),(10,3),(12,5),(24,5)\}.\]
For further results on generalized Petersen graphs we may refer the interested reader to Alspach \cite{ba}, and Castagna and Prins \cite{fc-gp}.

The aim of this paper is to consider the same problem by defining a new class of cubic graphs arising from the Tutte-Coxeter graph. A \textit{generalized Tutte-Coxeter graph} with respect to positive integers $n\geq3$ ($n$ even) and $1\leq k<n/2$, denoted by $TC(n,k)$, is defined on a set of $3n$ vertices 
\[a_0,a_1,\ldots,a_{n-1},b_0,b_1,\ldots,b_{n-1},c_0,c_1,\ldots,c_{n-1}\]
whose its edges are
\begin{flalign*}
&&\{a_i,a_{i+1}\},\{a_i,b_i\},\{b_i,b_{i+n/2}\},\{b_i,c_i\}\ \mbox{and}\ \{c_i,c_{i+k}\},&&
\end{flalign*}
where $i$ ranges over $\{1,2,\ldots,n\}$. Here, all the indices are taken modulo $n$. We note that  $T(10,3)$ is the well-known Tutte-Coxeter graph. The same as for generalized Petersen graph, it is natural to ask which generalized Tutte-Coxeter graphs are vertex-transitive or edge-transitive? We shall prove that the only vertex-transitive (edge-transitive) graph among generalized Tutte-Coxeter graphs is the Tutte-Coxeter graph giving rise to a characterization of the Tutte-Coxeter graph.
\section{The Generalized Tutte-Coxeter Graph}
We begin with naming the vertices and egdes of the generalized Tutte-Coxeter graph as follows:
\begin{itemize}
\item The sets $\{a_0,a_1,\ldots,a_{n-1}\}$, $\{b_0,b_1,\ldots,b_{n-1}\}$ and $\{c_0,c_1,\ldots,c_{n-1}\}$ of vertices denote the outer, middle and inner vertices, respectively.
\item The sets $\{\{a_i,a_{i+1}\}:1\leq i\leq n\}$,$\{\{a_i,b_i\}:1\leq i\leq n\}$, $\{\{b_i,b_{i+n/2}\}:1\leq i\leq n\}$, $\{\{b_i,c_i\}:1\leq i\leq n\}$ and $\{\{c_i,c_{i+k}\}:1\leq i\leq n\}$ of edges denote the outer edges, spokes of type 1, middle edges, spokes of type 2 and inner edges, respectively.
\end{itemize}

It is easy to see that, the subgraph induced by inner vertices is a union of $d$ disjoint $(n/d)$-cycles, where $d=\gcd(n,k)$.

For a given cycle $C$ of $TC(n,k)$, let $ov(C)$, $mv(C)$ and $iv(C)$ be the number of outer, middle and inner vertices, and let $oe(C)$, $s_1e(C)$, $me(C)$, $s_2e(C)$ and $ie(C)$ be the number of outer edges, spokes of type 1, middle edges, spokes of type 2 and inner edges of $C$, respectively. Let $\C_l$ be the set of all $l$-cycles of $TC(n,k)$ and put
\begin{flalign*}
&&OV(l)&=\sum_{C\in\C_l}ov(C),&OE(l)&=\sum_{C\in\C_l}oe(C),&&\\
&&MV(l)&=\sum_{C\in\C_l}mv(C),&S_1E(l)&=\sum_{C\in\C_l}s_1e(C),&&\\
&&IV(l)&=\sum_{C\in\C_l}iv(C),&ME(l)&=\sum_{C\in\C_l}me(C),&&\\
&&&&S_2E(l)&=\sum_{C\in\C_l}s_2e(C),&&\\
&&&&IE(l)&=\sum_{C\in\C_l}ie(C),&&
\end{flalign*}
for all $l\geq3$. It is clear that the graph $TC(n,k)$ is vertex-transitive only if $OV(l)=MV(l)=IV(l)$ for all $l\geq3$. Similarly, if the the graph $TC(n,k)$ is edge-transitive, then we must have $OE(l)=S_1E(l)=2ME(l)=S_2E(l)=IE(l)$ for all $l\geq3$. 

By evaluating the quantities $OV(l)$, $MV(l)$, $IV(l)$ and $OE(l)$, $S_1E(l)$, $ME(l)$, $S_2E(l)$, $IE(l)$, we will obtain all vertex-transitive and edge-transitive graphs among generalized Tutte-Coxeter graphs, respectively, when $l$ takes only the values $6$, $7$ and $8$.

First, we consider vertex-transitive graphs.
\begin{theorem}\label{vertex-transitive}
The graph $TC(n,k)$ is vertex-transitive if and only if $(n,k)=(10,3)$. 
\end{theorem}
\begin{proof}
If $(n,k)=(10,3)$, then $TC(n,k)$ is the Tutte-Coxeter graph and we are done. Now suppose that $(n,k)\neq(10,3)$. Clearly, $TC(n,k)$ has $8$-cycles of types $(11)$ and $(12)$ (see Table 1). If $k\geq3$, then $TC(n,k)$ has $8$-cycles of types different from $(11)$ and $(12)$ if and only if $(n,k)$ is of the given forms in Table 2, in which all $8$-cycles are described. In all cases, the equation $OV(8)=MV(8)=IV(8)$ is never satisfied so that $TC(n,k)$ is not vertex-transitive. Hence $k=1$ or $2$.

If $k=1$ then by invoking Table 1, one can easily see that the equation $OV(8)=MV(8)=IV(8)$ holds only if $n=8$. But then $T(8,1)$ has only two families of different $7$-cycles
\[\{a_i,a_{i+1},a_{i+2}, a_{i+3},a_{i+4},b_{i+n/2},b_i\}\]
and
\[\{c_i,c_{i+1},c_{i+2}, c_{i+3},c_{i+4},b_{i+n/2},b_i\}\]
for $i=1,\ldots,n$, from which it follows that $OV(7)=IV(7)=5n$ and $MV(7)=4n$. Hence, $TC(n,1)$ is not vertex-transitive.

Finally, if $k=2$ then one can easily see, by utilizing Table 3, that the equation $OV(7)=MV(7)=IV(7)$ is never satisfied so that $TC(n,2)$ is not vertex-transitive. Then proof is complete.
\end{proof}
\begin{theorem}\label{edge-transitive}
The graph $TC(n,k)$ is edge-transitive if and only if $(n,k)=(10,3)$. 
\end{theorem}
\begin{proof}
If $(n,k)=(10,3)$, then $TC(n,k)$ is the Tutte-Coxeter graph and so is edge-transitive. Hence assume that $(n,k)\neq(10,3)$. The same as in the proof of Theorem \ref{vertex-transitive}, it is easy to see that the equation $OE(8)=S_1E(8)=2ME(8)=S_2E(8)=IE(8)$ is never satisfied when $k\geq3$. Thus, we just consider the cases $k=1$ and $k=2$.

If $k=1$ then the only possible $6$-cycles in $TC(n,1)$ are $\{a_i,a_{i+1},b_{i+1},c_{i+1},c_i,b_i\}$  for $i=1,\ldots,n$ whenever $n\neq6$. Hence $OE(6)=IE(6)=n$, $ME(6)=0$ and $S_1E(6)=S_2E(6)=2n$, which imply that $TC(n,1)$ is not edge-transitive. In the case where $n=6$,  we have $8$-cycles of types $(1)$, $(9_+)$, $(9_-)$, $(10_+)$, $(10_-)$, $(11)$ and $(12)$ as illustrated in Table 1. Therefore $OE(8)=IE(8)=9n$, $ME(8)=6n$ and $S_2E(8)=S_1E(8)=4n$, and again $TC(n,1)$ is not edge-transitive.

Finally, assume that $k=2$. If $n\neq 6, 8, 14$ and $16$, then $TC(n,2)$ has only the $7$-cycles  $\{a_i,a_{i+1},a_{i+2},b_{i+2},c_{i+2},c_i,b_i\}$ for $i=1,\ldots,n$. This shows that $OE(7)=S_1E(6)=S_2E(6)=2n$, $ME(7)=0$ and $IE(7)=n$, hence $TC(n,2)$ is not edge-transitive. Invoking Table 3 in the cases where $n=6,8,14$ or $16$, it yields that $OE(7)\neq IE(7)$, that is, $TC(n,2)$ is not edge-transitive. The proof is complete.
\end{proof}

\hspace{-1.5cm}
\begin{sideways}
\begin{minipage}{18cm}
\begin{center}
Table 1: $k\neq2$\\
{\tiny\begin{tabular}{|c|c|c|c|ccc|ccccc|}
\hline
Type&$8$-Cycles&Conditions&\#&$ov$&$mv$&$iv$&$oe$&$s_1e$&$me$&$s_2e$&$ie$\\
\hline
$1$&$\{a_i,a_{i+1},a_{i+2},b_{i+2},c_{i+2},c_{i+1},c_i,b_i\}$&$k=1$&$n$&$3$&$2$&$3$&$2$&$2$&$0$&$2$&$2$\\
\hline
$2$&$\{a_i,a_{i+1},a_{i+2},a_{i+3},b_{i+3},c_{i+3},c_i,b_i\}$&$k=3$ or $n-k=3$&$n$&$4$&$2$&$2$&$3$&$2$&$0$&$2$&$1$\\
\hline
$3$&$\{a_i,a_{i+1},b_{i+1},c_{i+1},c_{(i+1)+k},c_{(i+1)+2k},c_i,b_i\}$&$n=3k+1$&$n$&$2$&$2$&$4$&$1$&$2$&$0$&$2$&$3$\\
\hline
$3'$&$\{a_{i+1},a_i,b_i,c_i,c_{i+k},c_{i+2k},c_{i+1},b_{i+1}\}$&$n=3k-1$&$n$&$2$&$2$&$4$&$1$&$2$&$0$&$2$&$3$\\
\hline
$4$&$\{a_i,a_{i+1},a_{i+2},b_{i+2},c_{i+2},c_{(i+2)+k},c_i,b_i\}$&$n=2k+2$&$n$&$3$&$2$&$3$&$2$&$2$&$0$&$2$&$2$\\
\hline
$5$&$\{a_i,a_{i+1},a_{i+2},a_{i+3},a_{i+4},a_{i+5},a_{i+6},a_{i+7}\}$&$n=8$&$1$&$8$&$0$&$0$&$8$&$0$&$0$&$0$&$0$\\
\hline
$6$&$\{c_i,c_{i+k},c_{i+2k},c_{i+3k},c_{i+4k},c_{i+5k},c_{i+6k},c_{i+7k}\}$&$n=8k$&$k$&$0$&$0$&$8$&$0$&$0$&$0$&$0$&$8$\\
\hline
$6'$&$\{c_i,c_{i+k},c_{i+2k},c_{i+3k},c_{i+4k},c_{i+5k},c_{i+6k},c_{i+7k}\}$&$3n=8k$&$n/8$&$0$&$0$&$8$&$0$&$0$&$0$&$0$&$8$\\
\hline
$7$&$\{a_i,a_{i+1},a_{i+2},a_{i+3},a_{i+4},a_{i+5},b_{i+5},b_i\}$&$n=10$&$n$&$6$&$2$&$0$&$5$&$2$&$1$&$0$&$0$\\
\hline
$8$&$\{b_i,c_i,c_{i+k},c_{i+2k},c_{i+3k},c_{i+4k},c_{i+5k},b_{i+n/2}\}$&$n=10k$&$n$&$0$&$2$&$6$&$0$&$0$&$1$&$2$&$5$\\
\hline
$8'$&$\{b_i,c_i,c_{i+k},c_{i+2k},c_{i+3k},c_{i+4k},c_{i+5k},b_{i+n/2}\}$&$3n=10k$&$n$&$0$&$2$&$6$&$0$&$0$&$1$&$2$&$5$\\
\hline
$9_+$&$\{a_i,a_{i+1},a_{i+2},b_{i+2},c_{i+2},c_{(i+2)+k)},b_{i+n/2},b_i\}$&$n=2k+4$&$n$&$3$&$3$&$2$&$2$&$2$&$1$&$2$&$1$\\
\hline
$9_-$&$\{a_i,a_{-(i+1)},a_{-(i+2)},b_{-(i+2)},c_{-(i+2)},c_{-((i+2)+k)},b_{i+n/2},b_i\}$&$n=2k+4$&$n$&$3$&$3$&$2$&$2$&$2$&$1$&$2$&$1$\\
\hline
$10_+$&$\{a_i,a_{i+1},b_{i+1},c_{i+1},c_{(i+1)+k},c_{(i+1)+2k},b_{i+n/2},b_i\}$&$n=4k+2$&$n$&$2$&$3$&$3$&$1$&$2$&$1$&$2$&$2$\\
\hline
$10_-$&$\{a_i,a_{-(i+1)},b_{-(i+1)},c_{-(i+1)},c_{-((i+1)+k)},c_{-((i+1)+2k)},b_{i+n/2},b_i\}$&$n=4k+2$&$n$&$2$&$3$&$3$&$1$&$2$&$1$&$2$&$2$\\
\hline
$10'_+$&$\{a_i,a_{i+1},b_{i+1},c_{i+1},c_{((i+1)-k)},c_{((i+1)-2k)},b_{i+n/2},b_i\}$&$n=4k-2$&$n$&$2$&$3$&$3$&$1$&$2$&$1$&$2$&$2$\\
\hline
$10'_-$&$\{a_i,a_{-(i+1)},b_{-(i+1)},c_{-(i+1)},c_{k-(i+1)},c_{2k-(i+1)},b_{i+n/2},b_i\}$&$n=4k-2$&$n$&$2$&$3$&$3$&$1$&$2$&$1$&$2$&$2$\\
\hline
$11$&$\{a_i,a_{i+1},b_{i+1},b_{(i+1)+n/2},a_{(i+1)+n/2},a_{i+n/2},b_{i+n/2},b_i\}$&$n\geq 4$&$n/2$&$4$&$4$&$0$&$2$&$4$&$2$&$0$&$0$\\
\hline
$12$&$\{b_i,c_i,c_{i+k},b_{i+k},b_{(i+k)+n/2},c_{(i+k)+n/2},c_{i+n/2},b_{i+n/2}\}$&$n\geq 4$&$n/2$&$0$&$4$&$4$&$0$&$0$&$2$&$4$&$2$\\
\hline
\end{tabular}}
\\\vspace{0.25cm}
Table 2: $k\neq1,2$\\
{\tiny\begin{tabular}{|c|c|c|ccc|ccccc|}
\hline
The Values $n$&The Values $k$&Type of $8$-circuits&$OV$&$MV$&$IV$&$OE$&$S_1E$&$ME$&$S_2E$&$IE$\\
\hline
$n=3k+1$&$k\neq3$&$(3)$,$(11)$,$(12)$&$4n$&$6n$&$6n$&$2n$&$4n$&$2n$&$4n$&$4n$\\
\hline
$n=3k-1$&$k\neq3,5$&$(3')$,$(11)$,$(12)$&$4n$&$6n$&$6n$&$2n$&$4n$&$2n$&$4n$&$4n$\\
\hline
$n=2k+2$&$k\neq3$&$(4)$,$(11)$,$(12)$&$5n$&$6n$&$5n$&$3n$&$4n$&$2n$&$4n$&$3n$\\
\hline
$n=8$&$k=3$&$(2)$,$(3')$,$(4)$,$(5)$,$(6')$,$(11)$,$(12)$&$12n$&$10n$&$12n$&$8n$&$8n$&$2n$&$8n$&$8n$\\
\hline
$n=8k$&$k\neq3$&$(6)$,$(11)$,$(12)$&$2n$&$4n$&$3n$&$n$&$2n$&$2n$&$2n$&$2n$\\
\hline
$n=8k$&$k=3$&$(2)$,$(6)$,$(11)$,$(12)$&$6n$&$6n$&$5n$&$4n$&$4n$&$2n$&$4n$&$3n$\\
\hline
$3n=8k$&$k\neq3,6$&$(6')$,$(11)$,$(12)$&$2n$&$4n$&$3n$&$n$&$2n$&$2n$&$2n$&$2n$\\
\hline
$n=10$&$k=4$&$(4)$,$(7)$,$(11)$,$(12)$&$11n$&$8n$&$5n$&$8n$&$6n$&$3n$&$4n$&$3n$\\
\hline
$n=10k$&$k\neq3$&$(8)$,$(11)$,$(12)$&$2n$&$6n$&$8n$&$n$&$2n$&$3n$&$4n$&$6n$\\
\hline
$n=10k$&$k=3$&$(2)$,$(8)$,$(11)$,$(12)$&$6n$&$8n$&$10n$&$4n$&$4n$&$3n$&$6n$&$7n$\\
\hline
$3n=10k$&$k\neq3$&$(8')$,$(11)$,$(12)$&$2n$&$6n$&$8n$&$n$&$2n$&$3n$&$4n$&$6n$\\
\hline
$n=2k+4$&$k\neq3,5,6$&$(9_-)$,$(9_+)$,$(11$),$(12)$&$8n$&$10n$&$6n$&$5n$&$6n$&$4n$&$6n$&$3n$\\
\hline
$n=14$&$k=5$&$(3')$,$(9_-)$,$(9_+)$,$(11)$,$(12)$&$10n$&$12n$&$10n$&$6n$&$8n$&$4n$&$8n$&$6n$\\
\hline
$n=16$&$k=6$&$(6')$,$(9_-)$,$(9_+)$,$(11)$,$(12)$&$8n$&$10n$&$7n$&$5n$&$6n$&$4n$&$6n$&$4n$\\
\hline
$n=4k+2$&$k\neq3$&$(10_-)$,$(10_+)$,$(11)$,$(12)$&$6n$&$10n$&$8n$&$3n$&$6n$&$4n$&$6n$&$5n$\\
\hline
$n=4k+2$&$k=3$&$(2)$,$(10_-)$,$(10_+)$,$(11)$,$(12)$&$10n$&$12n$&$10n$&$6n$&$8n$&$4n$&$8n$&$6n$\\
\hline
$n=4k-2$&$k\neq3$&$(10'_-)$,$(10'_+)$,$(11)$,$(12)$&$6n$&$10n$&$8n$&$3n$&$6n$&$4n$&$6n$&$5n$\\
\hline
\end{tabular}}
\\\vspace{0.25cm}
Table 3: $k=2$\\
{\tiny\begin{tabular}{|c|c|c|c|ccc|ccccc|}
\hline
Type&$7$-Cycles&Conditions&\#&$ov$&$mv$&$iv$&$oe$&$s_1e$&$me$&$s_2e$&$ie$\\
\hline
$1$&$\{a_i,a_{i+1},a_{i+2}, b_{i+2},c_{i+2},c_i,b_i\}$&$n=6,8,14$ or $16$&$n$&$3$&$2$&$2$&$2$&$2$&$0$&$2$&$1$\\
\hline
$2$&$\{a_i,a_{i+1},b_{i+1},c_{i+1},c_{i+3},b_{i+3},b_i\}$&$n=6$&$n$&$2$&$3$&$2$&$1$&$2$&$1$&$2$&$1$\\
\hline
$2'$&$\{a_i,a_{-(i+1)},b_{-(i+1)},c_{-(i+1)},c_{-(i+3)},b_{i+3},b_i\}$&$n=6$&$n$&$2$&$3$&$2$&$1$&$2$&$1$&$2$&$1$\\
\hline
$3$&$\{a_i,a_{i+1},a_{i+2},a_{i+3},a_{i+4},b_{i+4},b_i\}$&$n=8$&$n$&$5$&$2$&$0$&$4$&$2$&$1$&$0$&$0$\\
\hline
$4$&$\{c_i,c_{i+2},c_{i+4},c_{i+6},c_{i+8},c_{i+10},c_{i+12}\}$&$n=14$&$n/7$&$0$&$0$&$7$&$0$&$0$&$0$&$0$&$7$\\
\hline
$5$&$\{c_{i+1},c_{i+3},c_{i+5},c_{i+7},c_{i+9},c_{i+11},c_{i+13}\}$&$n=14$&$n/7$&$0$&$0$&$7$&$0$&$0$&$0$&$0$&$7$\\
\hline
$6$&$\{b_i, c_i,c_{i+2},c_{i+4},c_{i+6},c_{i+8},b_{i+8}\}$&$n=16$&$n$&$0$&$2$&$5$&$0$&$0$&$1$&$2$&$4$\\
\hline
\end{tabular}}
\end{center}
\end{minipage}
\end{sideways}
\end{document}